\documentclass[12pt]{amsart}
\usepackage{amsmath,amssymb,color,cite,verbatim}
\usepackage{etoolbox}
\apptocmd{\sloppy}{\hbadness 10000\relax}{}{}

\numberwithin{equation}{section}

\newtheorem{thm}[equation]{Theorem}

\newtheorem{lemma}[equation]{Lemma}
\newtheorem{cor}[equation]{Corollary}

\theoremstyle{definition}

\newcommand{\bP}{\mathbb{P}}
\newcommand{\C}{\mathbb{C}}

\DeclareMathOperator{\mult}{mult}

\newcommand{\abs}[1]{\lvert #1 \rvert}

\usepackage[pdftex,colorlinks,pagebackref, bookmarks=false]{hyperref}

\begin{document}

\title[Value sharing on a compact Riemann surface]{Shared values of meromorphic functions on a compact Riemann surface}

\author{Zhiguo Ding}
\address{
  Hunan Institute of Traffic Engineering,
  Heng\-yang, Hunan 421001 China
}
\email{ding8191@qq.com}

\author{Michael E. Zieve}
\address{
  Department of Mathematics,
  University of Michigan,
  530 Church Street,
  Ann Arbor, MI 48109-1043 USA
}
\email{zieve@umich.edu}
\urladdr{http://www.math.lsa.umich.edu/$\sim$zieve/}

\date{\today}

\begin{abstract}
We prove a new bound on the number of shared values of distinct meromorphic functions on a compact Riemann surface, explain a mistake in a previous paper on this topic, and give a survey of related questions.
\end{abstract}

\thanks{
The authors thank the journal's editors for their encouragement to write this paper.
The second author thanks the National Science Foundation for support under grant DMS-1601844.}

\maketitle


\section{Introduction} 

Let $X$ be a compact Riemann surface of genus $g$, and let $p$ and $q$ be distinct nonconstant meromorphic functions on $X$.
Let $S$ be the set of all points $\alpha$ in the Riemann sphere $\widehat \C$ for which the sets $p^{-1}(\alpha)$ and $q^{-1}(\alpha)$ are identical.  The main result of \cite{B} asserts that
\begin{equation}\label{B}
(\abs{S}-4)\cdot(\deg p + \deg q)\le 2g-2.
\end{equation}
However, as we will explain, and as the author of \cite{B} has confirmed via email, the argument in \cite{B} contains a crucial mistake so that it does not prove \eqref{B}; moreover, currently it is not known whether \eqref{B} is always true.  We will prove the following inequality, which is stronger than all inequalities in the literature that resemble \eqref{B} (other than \eqref{B} itself):

\begin{thm} \label{main}
We have
\begin{equation}\label{S}
(\abs{S}-2)\cdot\max(\deg p,\,\deg q)\le 2g-2+\deg p+\deg q.
\end{equation}
\end{thm}

Theorem~\ref{main} implies the following inequality from \cite{Schweizer}, which more closely resembles \eqref{B}:

\begin{cor}\label{cor}
We have
\begin{equation}\label{S0}
(\abs{S}-4)\cdot \max(\deg p,\,\deg q) \le 2g-2.
\end{equation}
\end{cor}

In this note we prove these results, explain the step in the proof in \cite{B} where the mistake occurs, and survey related topics.


\section{Proofs}
In this section we prove Theorem~\ref{main} and Corollary~\ref{cor}.  We maintain the notation from the first two sentences of the introduction.

\begin{lemma}\label{diff}
We have
\[
\deg(p-q)\le\deg p+\deg q.
\]
\end{lemma}

\begin{proof}
Every order-$k$ pole of $p-q$ must be a pole of order at least $k$ in at least one of $p$ or $q$, so the sum of the orders of the poles of $p-q$ is at most the sum of the orders of the poles of $p$ and $q$.  Since the sum of the orders of the poles of a meromorphic function on $X$ equals the degree of the function, the conclusion follows.
\end{proof}

\begin{proof}[Proof of Theorem~\ref{main}]
Let $T$ be a finite subset of $S$.  Replace $p,q,T$ by $\mu\circ p$, $\mu\circ q$, and $\mu(T)$ for a suitable M\"obius transformation $\mu(z)$, in order to assume that $\infty\notin T$.
The Riemann--Hurwitz formula for $p$ asserts that
\[
2g-2 = -2\deg p + \sum_{\beta\in X} (\mult_p(\beta)-1),
\]
where $\mult_p(\beta)$ is the local multiplicity of $p$ near $\beta$.  For any $\alpha\in\widehat\C$ we have $\sum_{\beta\in p^{-1}(\alpha)} \mult_p(\beta)=\deg p$, so that
\begin{equation}\label{rhp}
\begin{aligned}
2g-2 &\ge -2\deg p + \sum_{\alpha\in T}\sum_{\beta\in p^{-1}(\alpha)} (\mult_p(\beta)-1) \\
&= -2\deg p + \sum_{\alpha\in T} (\deg p-\abs{p^{-1}(\alpha)}) \\
&= -2\deg p + \abs{T}\deg p - \abs{p^{-1}(T)}.
\end{aligned}
\end{equation}
Likewise
\begin{equation}\label{rhq}
2g-2\ge -2\deg q + \abs{T}\deg q - \abs{q^{-1}(T)}.
\end{equation}
Since $p^{-1}(T)=q^{-1}(T)$ by hypothesis, the previous two inequalities imply that
\begin{equation}\label{proof}
2g-2 \ge (\abs{T}-2)\cdot\max(\deg p,\,\deg q) - \abs{p^{-1}(T)}.
\end{equation}
Since $p^{-1}(T)=q^{-1}(T)$ and $\infty\notin T$,
the set $p^{-1}(T)$ is contained in the set of zeroes of the meromorphic function $p-q$.  This function is nonzero by hypothesis, so it has at most as many zeroes as its degree.  Thus Lemma~\ref{diff} implies that
\[
\abs{p^{-1}(T)}\le\deg(p-q)\le \deg p + \deg q.
\]
Combining this with \eqref{proof} yields
\begin{equation}\label{proof2}
2g-2 \ge (\abs{T}-2)\cdot\max(\deg p,\,\deg q) - (\deg p + \deg q).
\end{equation}
Thus
\begin{equation}\label{pT}
\abs{T}\le 2 + \frac{2g-2+\deg p+\deg q}{\max(\deg p,\,\deg q)},
\end{equation}
so that in particular $T$ is bounded.
It follows that $S$ is finite, so we may choose $T$ to be $S$ in \eqref{proof2} in order to obtain \eqref{S}, which concludes the proof.
\end{proof}

\begin{proof}[Proof of Corollary~\ref{cor}]
Combining \eqref{S} with the inequality
\[
\deg p+\deg q\le 2\cdot\max(\deg p,\,\deg q)
\]
yields
\begin{align*}
(\abs{S}-2)\cdot\max(\deg p,\,\deg q)&\le 2g-2+\deg p+\deg q \\
&\le 2g-2+2\cdot\max(\deg p,\,\deg q),
\end{align*}
and subtracting $2\cdot\max(\deg p,\,\deg q)$ from both sides yields \eqref{S0}.
\end{proof}


\section{The mistake in \cite{B}}

We now compare the argument in \cite{B} with our proof of Theorem~\ref{main}.  The argument in \cite{B} also uses the Riemann--Hurwitz formula to deduce the inequalities \eqref{rhp} and \eqref{rhq}, and proves Lemma~\ref{diff} by a different argument than ours.  It then claims that \eqref{B} follows from these results and the inequality
\begin{equation}\label{ballico}
2g-2\ge -2\deg(p-q) + (\deg(p-q)-\abs{p^{-1}(T)}).
\end{equation}
Equations \eqref{rhp} and \eqref{rhq} have the form $\abs{T}\le A+B\cdot\abs{p^{-1}(T)}$ where $A$ and $B$ are constants depending on $g$ and the degrees of $p$ and $q$, with $B>0$.  Thus, in order to deduce an upper bound on $\abs{T}$ which depends only on $g$ and the degrees of $p$ and $q$, one needs an upper bound on $\abs{p^{-1}(T)}$.  However, \eqref{ballico} instead yields a lower bound on $\abs{p^{-1}(T)}$, 
and Lemma~\ref{diff} does not involve $\abs{T}$ or $\abs{p^{-1}(T)}$, 
so it is not possible to
obtain an upper bound on $\abs{T}$ of the desired form  by combining  \eqref{rhp}, \eqref{rhq}, \eqref{ballico} and Lemma~\ref{diff}.

It is reasonable to guess that the author of \cite{B} was attempting to produce a similar proof to the one we gave for \eqref{S} (except for his different proof of Lemma~\ref{diff}), since \cite{B} begins by saying that its proof was suggested by reading \cite{Sauerrat}, and the argument in \cite{Sauerrat} is the specialization to the case $X=\widehat \C$ of the argument used in the same author's previous paper \cite{Sauer} to prove a weaker version of Corollary~\ref{cor}.


\section{Related results}

We conclude with a quick survey of related literature, in which we continue to use the notation from the first two sentences of the introduction.

\begin{itemize}
\item Another bound is $\abs{S}\le 2+\sqrt{2g+2}$, which improves the best bound deducible from Corollary~\ref{cor} that depends only on $X$ and not on the functions $p$ and $q$ \cite{Schweizer}.

\item Stronger bounds on $\abs{S}$ are known for special classes of Riemann surfaces: for instance, if $X$ is hyperelliptic then $\abs{S}\le 5$ \cite{Schweizer}.

\item Stronger bounds on $\abs{S}$ are known when at least one element of $S$ has the same preimages \emph{counting multiplicities} under $p$ as it does under $q$ \cite{Sauer,Schweizer}.

\item Examples of $p,q,X$ where $S$ is ``large" appear in \cite{Pizer,Sauer,Sauerrat,Schweizer,Schweizer2}.

\item There are analogous results for holomorphic maps between prescribed compact Riemann surfaces \cite{Sauer}, and for holomorphic maps $p,q\colon X\to\bP^n(\C)$ having the same preimages of several hyperplanes \cite{RU,XR}.

\item Results about ``unique range sets" for meromorphic functions on $X$ appear in \cite{AW}; by definition, such a set is a finite subset $T$ of $\widehat \C$ for which any two distinct nonconstant meromorphic functions on $X$ have distinct preimages of $T$.

\item If there are three disjoint nonempty finite subsets $T_1,T_2,T_3$ of $\widehat \C$ such that each $T_i$ has the same preimages under $p$ (counting multiplicities) as it does under $q$, then $g\circ p=g\circ q$ for some nonconstant rational function $g(z)$ \cite{SZ}.

\item There are related results in case $X$ is either the complement of a finite subset of a compact Riemann surface  \cite{Schweizer2}, an open Riemann surface with a harmonic exhaustion  \cite{CZ,Schmid}, or a higher-dimensional algebraic variety \cite{D,L,L2}.

\end{itemize}



\end{document}